\documentclass[11pt]{amsart}
\usepackage{mathtools, amssymb, amsfonts, amsthm, enumitem}
\usepackage{epsfig, psfrag, color, multicol, graphicx, tikz,pdfsync}
\usepackage{amsfonts}
\usepackage{epstopdf}
\usepackage{xspace}

\usepackage{hyperref}

\usepackage[all]{xy}

\usetikzlibrary{calc}

\newtheorem{thm}{Theorem}[section]
\newtheorem{lemma}[thm]{Lemma}

\newtheorem{obs}[thm]{Observation}
\newtheorem{cor}[thm]{Corollary}

\theoremstyle{plain}

\theoremstyle{definition}

\newtheorem{rem}[thm]{Remark}

\numberwithin{equation}{section}

\def\bu{\bullet}

\def\sq{\square}

\def\zz{\mathbb Z}
\def\nn{\mathbb N}

\def\rr{\mathbb R}

\def\sm{\smallsetminus}

\def\Om{\Omega}

\def\la{\lambda}

\def\al{\alpha}
\def\be{\beta}

\def\vp{\varphi}

\def\ca{\mathcal A}

\def\cb{\mathcal B}
\def\cd{\mathcal D}

\def\ch{\mathcal H}

\def\cP{\mathcal P}
\def\cQ{\mathcal Q}

\def\cR{\mathcal R}

\def\cV{\mathcal V}

\def\ssuq{\subseteq}

\def\<{\langle}
\def\>{\rangle}
\def\La{\Lambda}

\def\SYT{ {\text {\rm SYT} } }
\def\IT{ {\text {\rm I\ts T} } }
\def\LE{ {\text {\rm LE} } }

\def\0{{\mathbf 0}}

\def\.{\hskip.06cm}
\def\ts{\hskip.03cm}

\def\bbw{\textbf{\textit{w}}}
\def\bbv{\textbf{\textit{v}}}

\def\bm{\textbf{\textit{m}}}

\def\br{{r}}

\def\bba{{\text{\bf a}}}
\def\bbb{{\text{\bf b}}}

\def\SP{{\textup{\textsf{\#P}}}}

\def\nin{\noindent}

\title[Hook inequalities]{Hook inequalities}
\date{}

\author[Igor Pak \and Fedor Petrov \and Viacheslav Sokolov]{Igor~Pak$^{\star}$, \ \, Fedor Petrov$^{\dagger}$, \  and  \  Viacheslav Sokolov$^{\ddagger}$}

\thanks{\thinspace ${\hspace{-.45ex}}^\star$Department of Mathematics,
UCLA, Los Angeles, CA, 90095. \,  Email: \texttt{pak@math.ucla.edu}}

\thanks{\thinspace ${\hspace{-.45ex}}^\dagger$Steklov Mathematical Institute, St.\ Petersburg, Russia. \,  Email: \texttt{fedyapetrov@gmail.com}}

\thanks{\thinspace ${\hspace{-.45ex}}^\ddagger$Steklov Mathematical Institute, St.\ Petersburg, Russia. \,   Email: \texttt{vi.soksok@gmail.com}}

\thanks{\thinspace \
\today}

\begin{document}

\begin{abstract}
We give an elementary proof of the recent \emph{hook inequality} given in~\cite{MPP-asy}:
$$
\prod_{u\in \lambda} \. h(u) \, \le \, \prod_{u\in \lambda} \. h^\ast(u),
$$
where $h(u)$ is the usual hook in Young diagram~$\la$, and $h^\ast(i,j)=i+j-1$.
We then obtain a large variety of similar inequalities and their high-dimensional generalizations.
\end{abstract}

\maketitle

\section*{Instroduction}

In Enumerative Combinatorics, the results are usually easy to state.
Essentially, you are counting the number of certain combinatorial
objects: exactly, asymptotically, bijectively or otherwise.
Judging the importance of the results is also relatively easy:
the more natural or interesting the objects are, and the stronger
or more elegant is the final formula, the better.  In fact,
the story or the context behind the results is usually superfluous
since they speak for themselves.  In the words of Gian-Carlo Rota,
one of the founding fathers of modern enumerative combinatorics:

\medskip

\begin{center}\begin{minipage}{12cm}%
\emph{``Combinatorics is an honest subject. $[$\ts..\ts$]$
You count balls in a box, and you either have the right number
or you haven't.''~\cite{Rota-Los-Alamos}
}\end{minipage}\end{center}

\medskip

\nin
It is only occasionally the context makes the difference; this paper
is an exception.

\smallskip

The subject of this paper are certain new combinatorial inequalities
for hook numbers of Young diagrams, rooted trees and their generalizations.
In two special cases these inequalities are known, have technical proof,
and arise in somewhat different areas.  Not only we give
elementary proofs of these results, we also setup a new framework
which allows us to obtain their advanced generalizations.  Viewed
by themselves, our most general inequalities are overwhelming the
senses -- they are just too far removed from anything the reader
would know and recognize as interesting (see Section~\ref{s:gen}).

\smallskip

We structure the paper in a way to postpone stating the most general
results until the end.  First, we tell the story of two combinatorial
inequalities which would ring a lot of bells for people in the area
(sections~\ref{s:diag} and~\ref{s:tree}).  We then introduce
combinatorial tools to prove the inequalities (\emph{majorization} and
\emph{shuffling}), see sections~\ref{s:maj} and~\ref{s:shuffle}.
We then gradually proceed with our generalizations, hoping not to
loose the reader in the process.

\bigskip

\section{First story: The number of standard Young tableaux}\label{s:diag}

\subsection{Hook-length formula}
Our story starts with the classical \emph{hook-length formula}~\cite{FRT}
for the number of \emph{standard Young tableaux} of a given shape.
This formula is ``\emph{a beautiful result in enumerative
combinatorics that is both mysterious and extremely well
studied}''~\cite{MPP1}.  We recall it first starting from
basic definitions.

Let \ts $\lambda=(\lambda_1,\ldots,\lambda_\ell)$, \ts
$\la_1\ge \ldots \ge \la_\ell> 0$ \ts
be an \emph{integer partition}. We say that $\la$ is a
\emph{partition of $n$}, write $\la \vdash n$ or $|\la|=n$,
if $\la_1+\la_2+\ldots+\la_\ell=n$.  Here $n$ is also called
the {\em size} of~$\la$.  Let $\lambda'$ denotes the
{\em conjugate partition} of $\lambda$, see Figure~\ref{f:diagram}.

\emph{Young diagram} of a partition $\la$ is a subset of
$\nn^2$ (visualized as the set of unit squares),
with $\la_1$ squares in the first row, $\la_2$ squares
in the second row, etc.  By tradition, the diagram is drawn
as in the figure, with the top left corner $(1,1)$.  For a square
$u=(i,j)$ of the Young diagram, we write $(i,j)\in \la$, where the first
coordinate increasing downward and the second from left to right.

\begin{figure}[hbt]
\begin{center}
\includegraphics[scale=0.74]{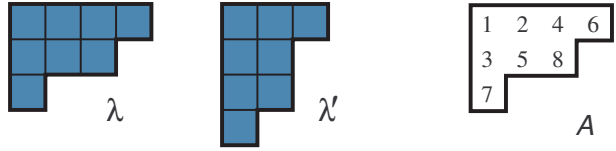}
\caption{Young diagram $\la=(4,3,1)$, conjugate diagram $\la'=(3,2,2,1)$, and
standard Young tableau $A\in \SYT(\la)$. }
\label{f:diagram}
\end{center}
\end{figure}

A {\em standard Young tableau} (SYT) of shape $\lambda$ is an
array~$A$ of shape $\lambda$ with the numbers $1,\ldots,n$, where
$n=|\lambda/\mu|$, each appearing once, increasing in rows
and columns (see Figure~\ref{f:diagram}).  Denote by $\SYT(\la)$
the set of standard Young tableaux of shape~$\la$. The number \ts
$\bigl|\SYT(\la)\bigr|$ \ts is fundamental in Algebraic Combinatorics;
it is equal to the dimension of the corresponding irreducible representation
of~$S_n$ and generalizes binomial coefficients, Catalan numbers, etc.
We refer to~\cite[Ch.~7]{EC} for the background and further references.

The \emph{hook length} $h_{\lambda}(i,j) := \la_i - i +\la_j' -j +1$
of a square $u=(i,j)\in [\la]$ is the number of squares directly
to the right or directly below~$u$, including~$u$, see Figure~\ref{f:hooks}.
The \emph{hook-length formula} states:
\begin{equation}
\label{eq:HLF} 
\bigl|\SYT(\la)\bigr| \, = \, n! \, \prod_{u\in \lambda} \. \frac{1}{h(u)}\.,
\end{equation}
(see~$\S$\ref{ss:finrem-hist}).  For example, for $\la=(4,3,1)$ as in the figure, we have:
$$
\bigl|\SYT(4,3,1)\bigr| \, = \, \frac{8!}{6\cdot 4 \cdot 4 \cdot 3 \cdot 2} \, = \. 70.
$$

\smallskip

\subsection{Hook inequalities}
We are now ready to state the first result which motivated much of our study.
Denote by $h^\ast(i,j)=i+j-1$ the \emph{anti-hook length} of the square $u=(i,j)\in\la$.
It is the number of squares directly to the left or directly above~$u$, including~$u$,
see Figure~\ref{f:hooks}.

\begin{thm}[{\cite[Prop.~12.1]{MPP-asy}}] \label{t:diag-hook}
For every Young diagram $\la$, we have:
$$
\prod_{u\in \lambda} \. h(u) \, \le \, \prod_{u\in \lambda} \. h^\ast(u).
$$
Moreover, the equality holds if and only if $\la$ has a rectangular shape.
\end{thm}

\begin{figure}[hbt]
\begin{center}
\includegraphics[scale=0.74]{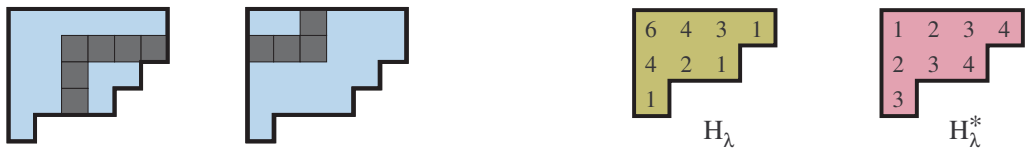}
\caption{Hook length $h(2,3)=6$, anti-hook length $h(2,3)=4$
for a partition $\la=(6,6,5,4,1)$.  Table of hook numbers \ts H$_\la=\bigl\{h(i,j)\bigr\}$
and table of anti-hook numbers \ts H$_\la^\ast=\bigl\{h^\ast(i,j)\bigr\}$ for $\la=(4,3,1)\vdash 8$. }
\label{f:hooks}
\end{center}
\end{figure}

For example, for $\la=(4,3,1)$ as in the figure, we have:
$$6\cdot 4 \cdot 4 \cdot 3 \cdot 2 \. = \. 576 \, \. \le \, \. 4 \cdot 4 \cdot 3 \cdot 3 \cdot 3 \cdot 2 \cdot 2 \. = \. 1728.
$$

\begin{rem}{\rm
The theorem is a corollary of a much more general inequality we prove
in~\cite[Thm~3.3]{MPP-asy}, which in turn is a corollary of the
\emph{Naruse hook-length formula} for the number of standard Young tableaux
of \emph{skew shape}.  This formula was discovered by Naruse in the
intersection of Representation Theory and Schubert Calculus;
see~\cite{MPP1,MPP2} for both algebraic and combinatorial proof and
further references.  We refer to~\cite{Swa} for a purely
combinatorial proof of Theorem~\ref{t:diag-hook},
from a very different point of view.
}\end{rem}

To really understand the above theorem, start with the following:

\begin{obs}\label{o:hook-diag}
For every Young diagram $\la$, we have:
$$
\sum_{u\in \lambda} \. h(u) \, = \, \sum_{u\in \lambda} \. h^\ast(u).
$$
\end{obs}

\begin{proof}
Note that both sides count the number of pairs of squares $x,y\in \la$,
such that $x$ is either directly above~$y$ or directly to the left of~$y$,
or $x=y$.
\end{proof}

There are two implications of this observation.  First, the theorem cannot
have a ``trivial'' proof by using monotonicity of anti-hooks over hooks.
Second, this suggests the distribution of hooks and anti-hooks have the
same mean, but the variance of hooks is larger than the variance
of anti-hooks.  That is, in fact, true and follows from our later
results:

\begin{cor}\label{c:diag-var}
For every Young diagram $\la$, we have:
$$
\sum_{u\in \lambda} \. h(u)^2 \, \ge \, \sum_{u\in \lambda} \. \bigl(h^\ast(u)\bigr)^2.
$$
\end{cor}

Here is yet another corollary from our general results.  For $(i,j)\in \la$, let
$$
a(i,j)\. := \. \bigl|\bigl\{(p,q)\in \la~:~i\le p, \. j \le q\bigr\}\bigr| \,
\quad \text{and} \quad a^\ast(i,j)\. := \. i\cdot j\ts.
$$
Here $a(i,j)$ is the area of~$\la$ below and to the right of~$(i,j)$.  Similarly,
$a^\ast(i,j)$ is the area of~$\la$ above and to the left of~$(i,j)$.  We refer to
these as \emph{area} and \emph{anti-area numbers}. We have the following equality
analogous to Observation~\ref{o:hook-diag} and the following analogue of
Theorem~\ref{t:diag-hook}.

\begin{cor}\label{c:diag-area}
For every Young diagram $\la$, we have:
$$
\prod_{u\in \lambda} \. a(u) \, \le \, \prod_{u\in \lambda} \. a^\ast(u).
$$
\end{cor}

The reader will have to wait until Section~\ref{s:maj}
to see how these results fit together.

\bigskip

\section{Second story: The number of increasing trees and linear extensions}\label{s:tree}

\subsection{Increasing trees}
Let $\tau$ be a \emph{rooted tree} on $n$ vertices, one of which is called
the \emph{root} $R\in V$.  We draw the trees as in Figure~\ref{f:tree}.
An \emph{increasing tree of shape~$\tau$} is a labeling of~$\tau$
with the numbers $1,\ldots,n$, each appearing once, increasing downwards,
away from~$R$.  Denote by $\IT(\tau)$ the set of increasing trees.

The number \ts $\bigl|\IT(\tau)\bigr|$ \ts famously has its own
analogue of the ``hook-length formula'' (see~$\S$\ref{ss:finrem-hist}).
For a vertex $v\in \tau$, denote by $b(v)$ the size of the \emph{branch} in~$\tau$
below~$v$.  We have:
\begin{equation}
\label{eq:HLF-tree} 
\bigl|\IT(\tau)\bigr| \, = \, n! \, \prod_{v\in \tau} \. \frac{1}{b(v)}\..
\end{equation}
For example, for $\tau$ as in the figure, we have:
$$
\bigl|\IT(\tau)\bigr| \, = \, \frac{8!}{8\cdot 4 \cdot 3 \cdot 2} \, = \. 210.
$$

\vskip-0.3cm
\begin{figure}[hbt]
\begin{center}
\includegraphics[scale=0.70]{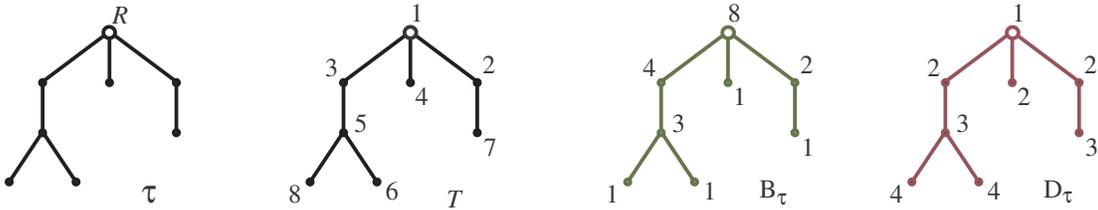}
\caption{Rooted tree $\tau$ with $n=8$ vertices, increasing tree $T$ of shape~$\tau$,
table of branch sizes \ts B$_\tau=\bigl\{b(v)\bigr\}$
and table of distances \ts D$_\tau =\bigl\{d(v)\bigr\}$. }
\label{f:tree}
\end{center}
\end{figure}

We can now state the inequality analogous to Theorem~\ref{t:diag-hook}.
For a vertex $v\in \tau$, denote by $d(v)$ the \emph{distance} from $v$
to the root~$R$, defined as the number of vertices on the shortest path
from $v$ to~$R$ (see Figure~\ref{f:tree}).

\begin{thm} \label{t:tree-hook}
For every rooted tree $\tau$, we have:
$$
\prod_{v\in \tau} \. b(v) \, \le \, \prod_{v\in \tau} \. d(v).
$$
Moreover, the equality holds if and only if $\tau$ is a path with
an endpoint at~$R$.
\end{thm}

For example, for $\tau$ as in the figure, we have:
$$
8\cdot 4  \cdot 3 \cdot 2 \. = \. 192 \, \. \le \, \. 4 \cdot 4 \cdot 3 \cdot 3 \cdot 2 \cdot 2 \cdot 2 \. = \. 1152.
$$

The theorem is easier than Theorem~\ref{t:diag-hook} and has a straightforward
proof by induction which we leave to the reader.  However, the following
direct analogue of Observation~\ref{o:hook-diag} is a warning
that it is still nontrivial:

\begin{obs}\label{o:hook-tree}
For every rooted tree $\tau$, we have:
$$
\sum_{v\in \tau} \. b(v) \, = \, \sum_{v\in \tau} \. d(v).
$$
\end{obs}

We omit the easy double counting proof of the observation.
Before we derive Theorem~\ref{t:tree-hook}, let us first state the analogue
of Corollary~\ref{c:diag-var}, which is of independent interest.

\begin{cor}\label{cor:tree-var}
For every rooted tree $\tau$, we have:
$$
\sum_{v\in \tau} \. b(v)^2 \, \ge \, \sum_{v\in \tau} \. d(v)^2.
$$
\end{cor}

Again, we derive the corollary from our general results.

\medskip

\subsection{Linear extensions}
Let $\cP$ be a ranked poset on a finite set $X$, $|X|=n$,
with linear ordering denoted by~$\prec$.  A \emph{linear extension}
of $\cP$ is a bijection $f: X\to \{1,\ldots,n\}$ such that
$f(x) < f(y)$ for all $x\prec y$.
Let $\LE(\cP)$ be the set of \emph{linear extensions} of~$\cP$.

Note that when $\cP$ is a Young diagram~$\la$ or a rooted tree~$\tau$ with
natural linear ordering, the number \ts $\bigl|\LE(\cP)\bigr|$ \ts
of linear extensions coincides with \ts $\bigl|\SYT(\la)\bigr|$ \ts and
\ts $\bigl|\IT(\tau)\bigr|$, respectively.  However, for general posets
the number \ts $\bigl|\LE(\cP)\bigr|$ \ts is hard to find both
theoretically and computationally. In fact, \ts $\bigl|\LE(\cP)\bigr|$ \ts is
\ts $\SP$-complete to compute~\cite{BW} (see also~\cite{DP}),
so no simple product formula is expected to exist in full generality
(cf.~\cite{Pro}).

Denote by $\br_\prec(x)$ the size the of the upper order ideal \ts
$\{y \in X~:~y\succeq x\}\subseteq\cP$, i.e.\ the number of elements
in~$\cP$ greater or equal to~$x$.  We have the following general inequality:

\begin{thm} [{\cite[Cor.~2]{HP}}]  \label{t:poset-ineq}
For every poset $\cP$, in the notation above, we have:
$$
\bigl|\LE(\cP)\bigr| \, \ge \, n!\. \prod_{x\in \cP} \. \frac{1}{\br_\prec(x)}\,.
$$
\end{thm}

\nin
This lower bound was proposed by Stanley~\cite[Exc.~3.57]{EC}
and proved by Hammett and Pittel in~\cite{HP} by an involved
probabilistic argument.  Heuristically, the theorem says that events
\ts $A_x=$ ``random bijection $f:X\to \{1,\ldots,n\}$ has $f(x)<f(y)$
for all $y\succ x$'' have positive correlation.

When $\cP$ is a rooted tree~$\tau$ on $n$ vertices, with the order~``$\prec$''
increasing down, the inequality in the theorem is an equality, and the sizes of
upper order ideal are exactly the branch size numbers: \ts $\br_\prec(v)=b(v)$.
When the order is reversed, we obtain the distance numbers: \ts
$\br_\succ(v)=d(v)$. We thus have:
$$
n!\. \prod_{v\in \tau} \. \frac{1}{b(v)} \, = \, \bigl|\IT(\tau)\bigr|
\, \ge \, n!\. \prod_{v\in \tau} \. \frac{1}{d(v)}\,,
$$
which implies the inequality in Theorem~\ref{t:tree-hook}.

\bigskip

\section{Majorization approach}\label{s:maj}

\subsection{Definitions and background}
Let \ts $\bba=(a_1,\ldots,a_n), \bbb=(b_1,\ldots,b_n)\in \rr^n$ \ts
be two non-increasing sequences \ts $a_1\ge \ldots \ge a_n >0$ \ts and \ts
$b_1\ge \ldots \ge b_n >0$. We say that $\bba$ \emph{majorizes}~$\bbb$,
write \ts $\bba \trianglerighteq \bbb$, if
\begin{equation}
\label{eq:maj}
\aligned
a_1 \ts + \ts \ldots \ts + \ts a_k \, &\ge \, b_1 \ts + \ts \ldots \ts + \ts b_k \quad
\text{for all} \ \. 1 \le k < n\ts, \ \ \text{and} \\
a_1 \ts + \ts \ldots \ts + \ts a_n \, &= \, b_1 \ts + \ts \ldots \ts + \ts b_n\ts.
\endaligned
\end{equation}
See~\cite{MOA} for a thorough treatment of majorization. 
We need the following important result:

\begin{thm}[Karamata's inequality] \label{t:karam}
Fix a strictly convex function \ts $\vp:\rr\to \rr$.  Let \ts $\bba, \bbb\in \rr^n$ \ts
be two sequences, s.t.\ $\bba \trianglerighteq \bbb$.  Then
$$
\sum_{i=1}^n \. \vp(a_i) \, \ge \, \sum_{i=1}^n \. \vp(b_i)\ts.
$$
Moreover, the equality holds if and only if \ts $\bba=\bbb$.
\end{thm}

This result is classical and the converse also holds, but we will not need it.
For streamlined proofs and further references see e.g.~\cite[$\S$3.17]{HLP}
and~\cite[$\S$28,$\S$30]{BB}.

\smallskip

Since our hooks numbers do not have a natural ordering, we will need
a multiset analogue of majorization.
For two multisets $\ca$ and $\cb$ with $n$~real numbers, let~$\bba$ and~$\bbb$
are non-increasing ordering of~$\ca$ and~$\cb$, respectively. We say that
$\ca$ \emph{majorizes}~$\cb$, write \ts $\ca \trianglerighteq \cb$, whenever
\ts $\bba \trianglerighteq \bbb$.  The following is a sufficient condition for
majorization.

\begin{lemma} \label{l:major-subsets}
Let $\ca,\cb$ two multisets with $n$~real numbers.  Suppose for all \ts
$\cb' \ssuq \cb$, there exist $\ca'\ssuq \ca$, such that \ts
$|\ca'|=|\cb'|$ \ts and
$$
\sum_{x\in \ca'} \. x \, \ge \, \sum_{y\in \cb'} \. y \,\ts.
$$
Then \ts $\ca \trianglerighteq \cb$.
\end{lemma}

The lemma is also standard and straightforward.  Since we use it repeatedly,
we include a simple proof for completeness.

\begin{proof}
Let~$\bba$ and~$\bbb$ be non-increasing ordering of~$\ca$ and~$\cb$, respectively.
Take \ts $\cb'=\{b_1,\ldots,b_k\}$ \ts and the corresponding \ts
$\ca'=\bigl\{a_{i_1},\ldots,a_{i_k}\bigr\}$, where \ts $1\le i_1< \cdots < i_k\le n$.  We have:
$$
b_1 \ts + \ts \ldots \ts + \ts  b_k \, \le \,  a_{i_1} \ts + \ts \ldots \ts + \ts a_{i_k}
\, \le \, a_1 \ts + \ts \ldots \ts + \ts a_k\ts,
$$
as desired.
\end{proof}

\medskip

\subsection{Applications to hook inequalities}
In notation of Section~\ref{s:diag}, we can now state the following new result.

\begin{thm}\label{t:diag-major}
Let $\la$ be a Young diagram. Denote by
$\ts\ch_\la = \bigl\{h(i,j), \ts (i,j)\in\la\bigr\}$ \ts
the multiset of hook numbers, and by \ts
$\ch^\ast_\la= \bigl\{h^\ast(i,j), \ts (i,j)\in\la\bigr\}$ \ts
be the multiset of anti-hook numbers.  Then \ts
$\ch_\la \trianglerighteq \ch^\ast_\la$.
\end{thm}

Here the equality part of~$(\divideontimes)$ is exactly Observation~\ref{o:hook-diag}.
The inequalities in~$(\divideontimes)$  will be deduced in the next section
from Lemma~\ref{l:major-subsets}.

Now, since $\vp(z)=z^2$ is strictly convex, Theorem~\ref{t:diag-major} and
Theorem~\ref{t:karam} imply Corollary~\ref{c:diag-var}.
Similarly, since $\vp(z)= -\log (z)$ is strictly convex on $\rr_{>0}$,
we obtain Theorem~\ref{t:diag-hook} from Theorem~\ref{t:diag-major} and
Theorem~\ref{t:karam} by taking logs on both sides.  The details are
straightforward.

\begin{thm}\label{t:area-major}
Let $\la$ be a Young diagram. Denote by \ts
$\ca_\la =\bigl\{a(i,j), \ts (i,j)\in\la\bigr\}$ \ts the multiset of area
numbers, and by \ts
$\ca^\ast_\la = \bigl\{a^\ast(i,j), \ts (i,j)\in\la\bigr\}$ \ts
the multiset of anti-area numbers.
Then \ts $\ca_\la \trianglerighteq \ca^\ast_\la$.
\end{thm}

By analogy with the argument above, obtain Corollary~\ref{c:diag-area} from
Theorem~\ref{t:area-major} and Theorem~\ref{t:karam} by taking logs on both sides.

\begin{thm}\label{t:tree-major}
Let $\tau$ be a rooted tree. Denote by \ts
$\cR_\tau =\bigl\{\br(v), \ts v\in\tau\bigr\}$ \ts the multiset of branch
numbers, and by \ts $\cd_\tau = \bigl\{d(v), \ts v\in\tau\bigr\}$ \ts
the multiset of distance numbers.
Then \ts $\cR_\tau \trianglerighteq \cd_\tau$.
\end{thm}

Now Corollary~\ref{cor:tree-var} follows from Theorem~\ref{t:tree-major}
and Theorem~\ref{t:karam} by taking $\vp(z)=z^2$, see the argument above.

\bigskip

\section{Shuffling in the plane}\label{s:shuffle}

\subsection{Hook majorization}  \label{s:shuffle-hooks}
For a square $(i,j)\in \la$, denote by $\al(i,j)=\la_i-j$ the
\emph{arm} at~$(i,j)$, and by $\be(i,j)=\la_j'-i$ the
\emph{leg} at~$(i,j)$. Similarly, denote by $\al^\ast(i,j)=j$
and $\be^\ast(i,j)=i$ the \emph{anti-arm} and \emph{anti-leg}.

\medskip

\nin
{\sc Shuffling in the plane.} \ts Let $X\ssuq \la$ be a subset of squares.
Perform two steps:


\nin
$(1)$ \. Push all squares in $X$ all the way up inside each column,


\nin
$(2)$ \. Push all squares in $X$ all the way left inside each row.

\begin{figure}[hbt]
\begin{center}
\includegraphics[scale=0.6]{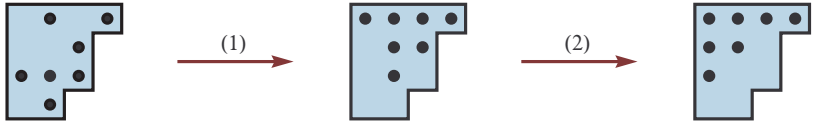}
\caption{Shuffling of $7$ squares in $\la=(4,3,3,2)$. }
\label{f:shuffle}
\end{center}
\end{figure}

\begin{proof}[Proof of Theorem~\ref{t:diag-major}]
Take $X=\{x_1,\ldots,x_k\}\ssuq \la$ and shuffle it as above.  After step~$(1)$, we
obtain set $X'=\{x_1',\ldots,x_k'\}\ssuq \la$, s.t.
\begin{equation}\label{eq:step-1}
\aligned
 \al(x'_i) \ge  \al(x_i), \quad \al^\ast(x'_i) = \al^\ast(x_i) \quad \text{for all} \ \ \.  1\le i\le k, \ \, \text{and} \\
 \be(x'_1) + \ldots \ts + \ts\be(x'_k) \. \ge \. \be^\ast(x_1)  +  \ldots  + \be^\ast(x_k)\ts. \hskip.5cm
\endaligned
\end{equation}
The first two inequalities hold by definition.  The last inequality hold
since all squares are pushed up, where the maximize the legs (cf.\ the
proof of Lemma~\ref{l:major-subsets}).

Similarly, after step~$(2)$ applied to $X'$, we
obtain set $Y=\{y_1,\ldots,y_k\}\ssuq \la$, s.t.
\begin{equation}\label{eq:step-2}
\aligned
\al(y_1) + \ldots +\al(y_k) \. \ge \. \al^\ast(x'_1) + \ldots +\al^\ast(x'_k), \ \, \text{and} \\
\be(y_i) \ge  \be(x_i), \quad \be^\ast(y_i) = \be^\ast(x'_i) \quad \text{for all} \ \ \.  1\le i\le k\ts.
\endaligned
\end{equation}
Combining equations~\eqref{eq:step-1} and~\eqref{eq:step-2} we obtain:
\begin{equation}\label{eq:step-12}
\aligned
\al(y_1) + \ldots +\al(y_k) \. \ge \. \al^\ast(x'_1) + \ldots +\al^\ast(x'_k) \. = \. \al^\ast(x_1) + \ldots +\al^\ast(x_k)\ts,\\
\be(y_1) \ts + \ts \ldots \ts +\ts \be(y_k) \. \ge \, \be(x'_1) \ts + \ts \ldots \ts + \ts \be(x'_k) \, \ge \. \be^\ast(x_1)
\ts + \ts \ldots \ts + \ts \be^\ast(x_k)\ts. \. \endaligned
\end{equation}
Using
$$h(i,j)\. = \. \al(i,j)+\be(i,j)+1\ts, \qquad h^\ast(i,j)\. = \. \al^\ast(i,j)+\be^\ast(i,j)+1
$$
we can sum two equations in~\eqref{eq:step-12} to obtain
\begin{equation}\label{eq:step-hook}
h(y_1) + \ldots + h(y_k) \. \ge \.  h^\ast(x_1) + \ldots + h^\ast(x_k)\ts.
\end{equation}
In summary, for every $X\ssuq \la$ there exists $Y\ssuq \la$, $|Y|=|X|$,
such that~\eqref{eq:step-hook} holds. Now by Lemma~\ref{l:major-subsets}
we obtain the result.
\end{proof}

\medskip

\subsection{Proof deconstruction}  \label{s:shuffle-deconstruction}
Before we proceed to several generalizations let us reformulate
the proof of Theorem~\ref{t:diag-major} in a different language.

Let \ts $h^\circ(i,j) = i + \la_j'-j = \al^\ast(i,j) + \be(i,j) -1$ \ts
be the \emph{semi-hook length}. For a subset $X$ of squares in~$\la$,
denote by $h(X)$, $h^\circ(X)$ and $h^\ast(X)$ the sums of hooks,
semi-hooks and anti-hooks numbers, respectively.

In the proof above, the sum of inequalities in~\eqref{eq:step-1}
and in~\eqref{eq:step-2}, gives
$$h^\ast(X) \. \le \. h^\circ(X) \quad \text{and} \quad
h^\circ(X) \. \le \. h(X),
$$
respectively. Combined, these give
\ts $h^\ast(X) \. \le \. h(X)$, a consequence of
inequalities in~\eqref{eq:step-12}.  This implies majorization
by Lemma~\ref{l:major-subsets}. \ $\sq$

\smallskip

Note that in the proof we split hooks and anti-hooks
into sums of arms/legs and anti-arms/anti-legs, respectively.
This is out of necessity as individually these parameters
are easier to control. The approach we take below to split
the generalized hooks into even more parameters.

\medskip

\subsection{Area majorization}  \label{s:shuffle-area}
We need new notation.
For a square $(i,j)\in \la$ and integers $p,q \in \zz$, denote
$$a_{pq}(i,j) \. := \. \left\{\aligned
& 1 \quad \text{if} \ \.\{ p, p+i\} \times \{ q,q+j\}   \subset \la\ts,\\
& 0 \quad \text{otherwise.}
\endaligned\right.
$$
Note that when $p$ and $q$ have the same sign, we have $a_{pq}(i,j)=1$
if and only if $(p+i,q+j)\in \la$.
%
Then
$$
a(i,j) \, := \. \sum_{(p,q) \in \nn^2} \. a_{pq}(i,j)\., \qquad
a^\ast(i,j) \, := \. \sum_{(p,q) \in \nn^2} \. a_{-p,-q}(i,j)\ts.
$$
Indeed, think of subscript $(p,q)$ as the vector shift.  Then the first sum is the number
of squares in~$\la$ to the right and below~$(i,j)$,  while the second sum is the number
of squares in~$\la$ to the left and above~$(i,j)$, as desired.

\begin{proof}[Proof of Theorem~\ref{t:area-major}]
We proceed as in the proof above.
Let
$$
a^\circ(i,j) \, := \. \sum_{(p,q) \in \nn^2} \. a_{p,-q}(i,j)\.
$$
For a subset of squares \ts $X\ssuq \la$ \ts denote by
$$
a_{pq}(X)\. := \. \sum_{(i,j)\in X} \. a_{p,q}(i,j), \quad
a_{pq}^\circ(X)  := \. \sum_{(i,j)\in X} \. a^\circ_{p,-q}(i,j), \quad
a_{pq}^\ast(X)  := \. \sum_{(i,j)\in X} \. a^\ast_{-p,-q}(i,j).
$$
Now consider the shuffling of squares \ts $X\to X' \to Y$ \ts
as in the proof above.  For every $(p,q)\in \nn$, we have:
$$
a_{pq}^\ast(X) \, \le \, a_{pq}^\circ(X') \, \le \, a_{pq}(Y)\ts.
$$
This follows simply by geometry of partitions: rows above are
larger or equal than rows below, and the columns to the left
are larger or equal than columns to the right.  Summing this
over all $(p,q)\in \nn$, we conclude:
$$
a^\ast(X) \, \le \, a^\circ(X') \, \le \, a(Y)\ts.
$$
This implies majorization by Lemma~\ref{l:major-subsets}.
\end{proof}

\medskip

\subsection{Generalization to weighted squares}  \label{s:shuffle-function}
We need new notation.  Fix a function \ts $g: \nn^2 \to \rr_+$.
For a square $(i,j)\in \la$, denote
$$\psi_{pq}(i,j) \. := \. \left\{\aligned
& g(p,q) \quad \text{if} \ \.(p+i,q+j)\in \la\ts,\\
& 0 \quad \text{otherwise.}
\endaligned\right.
$$
and
$$
\psi(i,j) \, := \. \sum_{(p,q) \in \nn^2} \. \psi_{pq}(i,j)\., \qquad
\psi^\ast(i,j) \, := \. \sum_{(p,q) \in \nn^2} \. \psi_{i-p,j-q}(p,q)\ts.
$$
Think of subscript $(p,q)$ as the vector shift.  Then $\psi(i,j)$ is a sum of function~$g$
over the fourth quadrant starting at~$(i,j)$,  while $\psi^\ast(i,j)$ is a sum of function~$g$
over the second quadrant starting at~$(i,j)$.

For example, when \ts $g(p,q)=1$ \ts we get the area and
anti-area numbers: \ts $\psi(i,j)=a(i,j)$ \ts
and \ts $\psi^\ast(i,j)=a^\ast(i,j)$.  Similarly, let
\begin{equation}\label{eq:function-g-hook}
g(p,q) \. := \. \left\{\aligned
& 1 \quad \text{if} \  \ p=0 \ \, \text{or} \ \, q=0,\\
& 0 \quad \text{otherwise.}
\endaligned\right.
\end{equation}
Then we get the hook and anti-hook numbers in this case: \ts $\psi(i,j)=h(i,j)$ \ts
and \ts $\psi^\ast(i,j)=h^\ast(i,j)$.

\begin{thm}\label{t:diag-function}
Let $\la$ be a Young diagram and fix \ts $g: \nn^2 \to \rr_+$. Denote by
$$
\Psi\langle\la,g\rangle \. := \, \bigl\{\psi(i,j), \. (i,j)\in\la\bigr\} \quad \text{and}
\quad \Psi^\ast\langle\la,g\rangle \. := \,  \bigl\{\psi^\ast(i,j), \. (i,j)\in\la\bigr\}
$$
the multisets of numbers defined above.  Then \ts
$\Psi\langle\la,g\rangle \trianglerighteq \Psi^\ast\langle\la,g\rangle$.
\end{thm}

This result is a common generalization of both Theorem~\ref{t:area-major}
and Theorem~\ref{t:diag-major}.  In the examples above, when
when \ts $g(p,q)=1$ \ts we have $\Psi\langle\la,g\rangle=\ca_\la$ and
$\Psi^\ast\langle\la,g\rangle=\ca^\ast_\la$.  Similarly, for $g$ as
in~\eqref{eq:function-g-hook}, we have $\Psi\langle\la,g\rangle=\ch_\la$ and
$\Psi^\ast\langle\la,g\rangle=\ch^\ast_\la$.  The proof follows verbatim
the proof of Theorem~\ref{t:area-major}, we omit the details.

\bigskip

\section{Generalized shuffling}\label{s:space}

\subsection{Hooks in space}  \label{s:space-hooks}
In $\rr^3$, \emph{solid partitions}~$\La$ are defined as lower
order ideals of the poset $\bigl(\nn^3, \prec\bigr)$, with the natural
ordering making $(0,0,0)$ the smallest.  We refer to ``up'' and ``down''
axis parallel direction in $\rr^3$ if it increases of decreases
coordinates, respectively.

Note that solid partitions generalize Young diagrams and
can be viewed as arrangements of unit cubes, such that whenever
$(i,j,k)\in \La$ we also have $(p,q,r)\in \La$ for all $p\le i$,
$q\le j$, and $r\le k$.

Now, there are two natural notions of the hooks: 1- and 2-dimensional.
More precisely, let
$$\aligned
R(i,j,k) \. & := \. \bigl|\bigl\{(p,j,k)\in \La~:~i\le p\bigr\}  \cup
\bigl\{(i,q,k)\in \La~:~j\le q\bigr\}  \cup
\bigl\{(i,j,r)\in \La~:~k\le r\bigr\}\bigr| \\
Q(i,j,k) \. & := \. \bigl|\bigl\{(p,q,k)\in \La~:~i\le p,\ts j\le q\bigr\}  \cup
\bigl\{(p,j,r)\in \La~:~i\le p, \ts k \le r\bigr\} \\
& \hskip7.cm \cup
\bigl\{(i,q,r)\in \La~:~j\le q, \ts k\le r\bigr\}\bigr| \endaligned
$$
These are subsets of squares of~$\La$ along the 1-dimensional rays or
2-dimensional quadrants emanating from a cube $(i,j,k)\in \La$.  The
anti-hooks $R^\ast(i,j,k)$ and $Q^\ast(i,j,k)$ are defined analogously,
with the direction of rays/quadrants reversed.

\begin{thm}\label{t:space-major}
Let $\La$ be a solid partition. Denote by $\cR_\La$, $\cQ_\La$,
$\cR^\ast_\La$, and $\cQ^\ast_\La$ the multisets of 1- and 2-dimensional
hook numbers and anti-hook numbers defined above.  Then \ts
$\cR_\La \trianglerighteq \cR^\ast_\La$ \ts and \ts $\cQ_\La
\trianglerighteq \cQ^\ast_\La$.
\end{thm}

By analogy with the arguments in Section~\ref{s:maj}, we obtain the
following generalizations of Theorem~\ref{t:diag-hook}.

\begin{cor}\label{c:space-hook}
Let $\La$ be a solid partition. Then:
$$
\prod_{(i,j,k)\in \La} \.  R(i,j,k) \, \le \prod_{(i,j,k)\in \La} \.  R^\ast(i,j,k), \quad \text{and} \quad
\prod_{(i,j,k)\in \La} \.  Q(i,j,k) \, \le \prod_{(i,j,k)\in \La} \.  Q^\ast(i,j,k).
$$
\end{cor}

There is also volume and anti-volume numbers, generalizing area and anti-area numbers:
$$
\aligned
V(i,j,k) \. & := \. \bigl|\bigl\{(p,q,r)\in \La~:~i\le p, \ts j\le q, \ts k \le r \bigr\} \bigr|\ts, \\
V^\ast(i,j,k) \. & := \. \bigl|\bigl\{(p,q,r)\in \La~:~p\le i, \ts q\le j, \ts r \le k\bigr\} \bigr|\ts.
\endaligned
$$

\begin{thm}\label{t:space-vol}
Let $\La$ be a solid partition. Denote by $\cV_\La$ and $\cV^\ast_\La$
the multisets of volume and anti-volume numbers defined above.  Then \ts
$\cV_\La \trianglerighteq \cV^\ast_\La$.
\end{thm}

This is a generalization of Theorem~\ref{t:area-major}.  By the same argument as above, we have
the following natural result generalizing Corollary~\ref{c:diag-area}.

\begin{cor}\label{c:space-hook}
Let $\La$ be a solid partition. Then:
$$
\prod_{(i,j,k)\in \La} V(i,j,k) \, \le \, \prod_{(i,j,k)\in \La} V^\ast(i,j,k)\..
$$
\end{cor}

This corollary has a curious implication.  By Theorem~\ref{t:poset-ineq}, we have two lower bounds
\begin{equation}\label{eq:le-curious}
\bigl|\LE(\cP_\La)\bigr| \, \ge \, n!\. \prod_{(i,j,k)\in \La} \frac{1}{ \. V(i,j,k)}\,, \quad \text{and} \quad
\bigl|\LE(\cP_\La)\bigr| \, \ge \, n!\. \prod_{(i,j,k)\in \La} \frac{1}{ \. V^\ast(i,j,k)}\,,
\end{equation}
where $\cP_\La$ is a poset defined~$\La$ as a subposet of~$\nn^3$. Corollary~\ref{c:space-hook} implies
that first of the  bounds in~\eqref{eq:le-curious} is always sharper than the second.
Of course, when $\La$ is a box these lower bounds coincide.

\medskip

\subsection{Shuffling in space}  \label{s:space-hooks} We give a brief outline
of the proof of theorems above.  First, define the shuffling in the space as follows.

\smallskip

\nin
{\sc Shuffling in $\nn^3$.} \ts Let $X\ssuq \La$ be a subset of cubes.
Perform three steps:

\nin
$(1)$ \. Push all cubes in $X$ all the way down along the $x$ axis,

\nin
$(2)$ \. Push all cubes in $X$ all the way down along the $y$ axis,

\nin
$(3)$ \. Push all cubes in $X$ all the way down along the $z$ axis.

\smallskip

\nin
Let us concentrate on the 1-dimensional hook numbers.
We introduce two intermediate hook numbers: $R^\circ(i,j,k)$ and
$R^{\circ\circ}(i,j,k)$ with two rays pointing up along the $y,z$ axes,
and one rays pointing up along the $z$~axis, respectively.  We then prove that
$$
\sum_{(i,j,k)\in X} \. \cR^\ast(i,j,k) \, \le \,
\sum_{(i,j,k)\in X'} \. \cR^\circ(i,j,k) \, \le \,
\sum_{(i,j,k)\in X''} \. \cR^{\circ\circ}(i,j,k) \, \le \,
\sum_{(i,j,k)\in Y} \. \cR(i,j,k),
$$
where \ts $X\to_{(1)} X'\to_{(2)} X''\to_{(3)} Y$ \ts is the result
of making shuffling steps.  Finally, by Lemma~\ref{l:major-subsets}
we obtain the first part of Theorem~\ref{t:space-major}.  We omit the details.

\medskip

\subsection{Back to trees}\label{s:space-trees}
Now that the reader is exhausted by 3-dimensional generalizations
of the hooks, the reader can restart anew with the following:

\smallskip
\nin
{\sc Shuffling in trees.} \ts Let $X\ssuq \tau$ be a subset of vertices and
let $w\in X$ be a vertex with the smallest distance $d(w)$.
Perform the following steps:

\nin
$\bu$ \. Move $w$ into~$R$.  Now do the shuffling in each branch
in $\tau\sm R$ by induction.


\smallskip

\begin{proof}[Sketch of proof of Theorem~\ref{t:tree-major}]
For a set of vertices $X\ssuq \tau$, denote by $Y$ the result of the shuffling.
We need to show that
$$
\sum_{v \in Y} \. d(v) \, \le \, \sum_{v \in X} \. \br(v)  \ts.
$$
Observe that $\br(R)=|\tau| \ge d(w)$, so moving $w$ to the root
is always advantageous for the sum of branch numbers. We continue
by induction leaving the details to the reader.
\end{proof}

\bigskip

\section{Greatest generalization}\label{s:gen}

Let us now state the most general version of the hook inequalities that
we promised in the introduction.  Denote by $T_1,\ldots,T_d$ infinite
locally finite rooted trees.  We view each $T_i$ as a poset with ordering
towards the root~$R_i$, and denote by $d_i$ the distance function
to~$R_i$ in the tree~$T_i$.  Consider a poset \ts
$\cP: =T_1\times \ldots \times T_d$ \ts with the natural ordering
towards the root, and let $\Omega$ be finite a lower ideal in~$\cP$.

Let $g:\nn^d\to \rr_+$ be a fixed weight function.  We say that
elements $\bbv=(v_1,\ldots,v_d)$ and $\bbw(w_1,\ldots,w_d)$ in $\Omega$ have
$\bm=(m_1,\ldots,m_d)$ shift, if $v_i\preccurlyeq w_i$ and $d(w_i)-d(v_i)=m_i$
for all~$i$.  In this case we write $\bbv\to_\bm \bbw$.  We can now define
the hook and anti-hook numbers
$$\aligned
H(\bbv) \. &:=\. \sum_{\bm\in \nn^d} \. g(\bm) \.\sum_{\bbw\in \Om} \. 1_{\{\bbv\to_\bm \bbw\}}\,, \quad \text{and}\\
H^\ast(\bbv) \. & :=\. \sum_{\bm\in \nn^d} \. g(\bm) \.\sum_{\bbw\in \Om} \. 1_{\{\bbw\to_\bm \bbv\}}\,.
\endaligned
$$

\begin{thm}\label{t:max-major}
Let $\Om$ be a lower order ideal in the product of trees
$T_1\times \ldots \times T_d$ poset.
Fix $g:\nn^d\to \rr_+$.  Denote by
$\ts\ch_{\Om,g} := \bigl\{H(\bbv), \ts \bbv\in\Om\bigr\}$ \ts and
$\ts\ch^\ast_{\Om,g} := \bigl\{H^\ast(\bbv), \ts \bbv\in\Om\bigr\}$ \ts
the multisets of hook and anti-hook numbers defined above.  Then \ts
$\ch_\Om \trianglerighteq \ch^\ast_\Om$.
\end{thm}

We leave to the reader to see how this result generalizes all previous
results, and what inequalities follow by Theorem~\ref{t:karam}.
The shuffling in the product of trees is also straightforward:
first shuffle in $T_1$, then in $T_2$, etc.  We omit the proof.

\bigskip

\section{Final remarks}\label{s:fin-rems}

\subsection{}  There is no universal agreement on the term
``anti-hook''.  While we follow e.g.~\cite{Jon}, the term
``cohook'' is also frequently used, see e.g.~\cite{BLRS,Sul}.
Since both ``cohook'' and ``anti-hook'' are also
used in other contexts, we opted for latter due
to personal preferences.

\subsection{} \label{ss:finrem-hist}
The hook-length formula~\eqref{eq:HLF} has a number
of different proofs (bijective, probabilistic, analytic, etc.)
as well as generalizations different from Naruse's.
The literature is too numerous to include; see~\cite[$\S$6.2]{CKP}
for a brief survey, and~\cite[Ch.~7]{EC} for further references.

The hook-length formula for trees~\eqref{eq:HLF-tree} is due to
Knuth~\cite[$\S$5.1.4, Exc.~20]{Knu}, and can be easily proved by
induction via removing the root.  Probabilistic and bijective
proofs are given in~\cite{SY} and~\cite{Bea}, respectively.

\subsection{}
Both hook-length formulas~\eqref{eq:HLF} and~\eqref{eq:HLF-tree}
have well understood $q$-analogues in terms of the numbers of
certain inversions of the standard Young tableaux and increasing
trees, respectively (see e.g.~\cite{CKP,EC,MPP1}).  Since
$\vp(n):=(n)_q=1+q+\ldots + q^{n-1}=(1-q^n)/(1-q)$ is logarithmically concave
in $n$, this gives
a $q$-deformation of Theorem~\ref{t:diag-hook}.  It would be
interesting to see if our hook inequalities have natural
$q$-analogues as polynomials in~$q$ (i.e.\ for each coefficient),
cf.~\cite[$\S$2]{Bre}.

\subsection{}
For solid partitions, there is no known close formula for the number of
linear extensions.  However, in this case both
lower bounds~\eqref{eq:le-curious} are not very sharp and weaker
than other, more elementary bounds (see~\cite{MPP-asy}).

\subsection{}
Despite their origin, in some ways our hooks inequalities have a
different nature than the inequalities usually studied in Enumerative
and Algebraic Combinatorics, whose proofs are non-robust and rely on
the global algebraic structure (see e.g.~\cite{Bra,Pak}).  Instead, they
resemble many isoperimetric and other geometric inequalities,
where the proofs are obtained by incremental improvements,
sometimes disguised by the variational principle (see e.g.~\cite{BZ}).
In fact, it is not hard to give an analytic generalization of
Theorem~\ref{t:diag-function} to the case when $\la$ is replaced
by a monotone curve, and $g$ is replaced by a Lebesgue measurable
non-negative function.  It would be interesting to see if there
are any new applications of this generalization.

\subsection{}
Swanson writes:  ``It would be interesting to find a
bijective explanation of \ts [Thm~\ref{t:diag-hook}]''~\cite{Swa}.
The methods of this paper are elementary, combinatorial in nature,
but fundamentally non-bijective.  They are likely to be the most
``bijective explanation'' possible in this case.  In fact,
we believe that our shuffling approach is truly a
``proof from the book''.

\medskip

\subsection*{Acknowledgements}
We are grateful to Joshua Swanson for telling us about~\cite{Swa}.
The first author is thankful to Alejandro Morales and Greta Panova for numerous
interesting conversations about the Naruse hook-length formula.
We are also grateful to BIRS in Banff, Canada, for hosting the
first two authors at the \emph{Asymptotic Algebraic Combinatorics} workshop
in March~2019, when this paper was finalized.  This collaboration started
when the first author asked a question on {\tt MathOverflow}.\footnote{See \ts
\href{http://mathoverflow.net/q/243846}{http://mathoverflow.net/q/243846}.}
The first author was partially
supported by the NSF. The second and the third authors were partially
supported by the RSCF (grant 17-71-20153).

\vskip.9cm



\begin{thebibliography}{BNRRR1}

\bibitem[BLRS]{BLRS}
K.~Barrese, N.~Loehr, J.~Remmel and B.~E.~Sagan, $m$-level rook placements,
\emph{J.~Combin. Theory, Ser.~A}~\textbf{124} (2014), 130--165.

\bibitem[Be\'a]{Bea}
B.~Be\'{a}ta,
Bijective proofs of the hook formula for rooted trees,
\emph{Ars Combin.}~\textbf{106} (2012), 483--494.

\bibitem[BB]{BB}
E.~Beckenbach and R.~Bellman,
\emph{Inequalities}, Springer, Berlin, 1961

\bibitem[Br\"{a}]{Bra}
P.~Br\"{a}nd\'{e}n, Unimodality, log-concavity, real-rootedness and beyond, in
\emph{Handbook of Enumerative Combinatorics}, CRC Press, Boca Raton, FL, 2015,
437--483; \ts {\tt arXiv:1410.6601}.

\bibitem[Bre]{Bre}
F.~Brenti,
Log-concave and unimodal sequences in algebra, combinatorics, and geometry: an update,
in \emph{Jerusalem Combinatorics}, AMS, Providence, RI, 1994, 71--89

\bibitem[BW]{BW}
G.~Brightwell and P.~Winkler,
Counting linear extensions,
\emph{Order}~\textbf{8} (1991), 225--247.

\bibitem[BZ]{BZ}
Yu.~D.~Burago and V.~A.~Zalgaller,
\emph{Geometric inequalities}, Springer, Berlin, 1988.

\bibitem[CKP]{CKP}
I.~Ciocan-Fontanine, M.~Konvalinka and I.~Pak,
The weighted hook length formula,
\emph{J.~Combin.\ Theory, Ser.~A}~\textbf{118} (2011), 1703--1717.

\bibitem[DP]{DP}
S.~Dittmer and I.~Pak, Counting linear extensions of restricted posets,
preprint (2018), 33 pp.; \ts {\tt arXiv:}{1802.06312}.

\bibitem[FRT]{FRT}
J.~S. Frame, G.~de~B. Robinson and R.~M. Thrall,
The hook graphs of the symmetric group,
{\em Canad.\ J.\ Math.}~\textbf{6} (1954), 316--324.

\bibitem[HP]{HP}
A.~Hammett and B.~Pittel,  How often are two permutations comparable?
\emph{Trans.\ AMS}~\textbf{360} (2008), 4541--4568.

\bibitem[HLP]{HLP}
G.~H.~Hardy, J.~E.~Littlewood and G.~P\'olya, \emph{Inequalities} (Second ed.),
Cambridge Univ.\ Press, 1952.

\bibitem[Jon]{Jon}
B.~F.~Jones,  Singular Chern classes of Schubert varieties via small resolution,
\emph{Int.\ Math.\ Res.\ Not.}~\textbf{2010} (2010), 1371--1416.

\bibitem[Knu]{Knu}
D.~Knuth, \emph{The Art of Computer Programming},
Vol.~III, Addison-Wesley, Reading, MA, 1973.


\bibitem[MOA]{MOA}
A.~Marshall, I.~Olkin and B.~C.~Arnold,
\emph{Inequalities: theory of majorization and its applications} (Second ed.),
Springer, New York, 2011.

\bibitem[MPP1]{MPP1}
A.~H.~Morales, I.~Pak and G.~Panova, Hook formulas for skew shapes~I.
$q$-analogues and bijections, {\em J.~Combin.\ Theory, Ser.~A}~\textbf{154}
(2018), 350--405.

\bibitem[MPP2]{MPP2}
A.~H.~Morales, I.~Pak and G.~Panova,
Hook formulas for skew shapes~II. Combinatorial proofs and enumerative
applications, {\em SIAM J.\ Discrete Math.}~\textbf{31} (2017), 1953--1989.

\bibitem[MPP3]{MPP-asy}
A.~H.~Morales, I.~Pak and G.~Panova,
Asymptotics for the number of standard Young tableaux of skew shape,
{\em European J.\ Combin.}~\textbf{70} (2018), 26--49.

\bibitem[Pak]{Pak}
I.~Pak, Combinatorial inequalities, \emph{Notices of the AMS}, to appear;
available at \url{https://tinyurl.com/y26efyoj}.

\bibitem[Pro]{Pro}
R.~A.~Proctor, $d$-Complete posets generalize Young diagrams for the hook
product formula, in \emph{RIMS K\^oky\^uroku}~\textbf{1913} (2014), 120--140.

\bibitem[Rota]{Rota-Los-Alamos}
G.-C.~Rota and D.~Sharp, Mathematics, Philosophy and Artificial Intelligence,
dialogue in \emph{Los Alamos Science}, No.~12 (Spring/Summer 1985), 94--104.

\bibitem[SY]{SY}
B.~E.~Sagan and Y.~N.~Yeh,
Probabilistic algorithms for trees,
\emph{Fibonacci Quart.}~\textbf{27} (1989), 201--208.

\bibitem[Sta]{EC}
R.~P.~Stanley, {\em Enumerative Combinatorics}, vol.~1 (Second ed.)
and~2, Cambridge Univ.\ Press, 2012 and~1999.

\bibitem[Sul]{Sul}
R.~Sulzgruber,
Symmetry properties of the Novelli-Pak-Stoyanovskii algorithm, in
\emph{Proc.\ 26th FPSAC (Chicago, USA)}, DMTCS, Nancy, 2014, 205--216.

\bibitem[Swa]{Swa}
J.~P.~Swanson, On the existence of tableaux with given modular major index,
\emph{Algebraic Combinatorics}~\textbf{1} (2018), 3--21.

\end{thebibliography}
\end{document}